\documentclass[9pt,journal]{IEEEtran}

\usepackage{amsmath,amsfonts,amssymb,mathtools}
\usepackage{hhline}
\usepackage{cite}
\usepackage{graphicx}
\usepackage{psfrag}
\usepackage{subfigure}
\usepackage{stfloats}
\usepackage{array}
\DeclareMathAlphabet\mathbfcal{OMS}{cmsy}{b}{n}

\newtheorem{Prop}{Proposition}

\newtheorem{exmp}{Example}

\newcommand{\tr}[1]{\mathop{{\rm \bf tr}\left[#1\right]}\nolimits}

\newcommand{\Di}[1]{\mathop{\partial_i #1}}                  
\newcommand{\DDi}[1]{\mathop{\left( \partial_i #1 \right)}}  

\usepackage{clrscode}
\makeatletter
\g@addto@macro\code@init\normalfont
\makeatother

\begin{document}
\title{UD-based Pairwise and MIMO Kalman-like Filtering for Estimation of Econometric Model Structures}

\author{M.V.~Kulikova and J.V.~Tsyganova and G.Yu. Kulikov
\thanks{Manuscript received ??; revised ??.  The first and third authors acknowledge the financial support of the Portuguese FCT~--- \emph{Funda\c{c}\~ao para a Ci\^encia e a Tecnologia}, through the projects UIDB/04621/2020 and UIDP/04621/2020 of CEMAT/IST-ID, Center for Computational and Stochastic Mathematics, Instituto Superior T\'ecnico, University of Lisbon.}
\thanks{The first and third authors are with
CEMAT, Instituto Superior T\'{e}cnico, Universidade de Lisboa,
          Av. Rovisco Pais 1,  1049-001 LISBOA, Portugal. The second author is with Ulyanovsk State University,  Str. L. Tolstoy 42, 432017 Ulyanovsk, Russian Federation. Emails: maria.kulikova@ist.utl.pt; TsyganovaJV@gmail.com; gkulikov@math.ist.utl.pt}}

\markboth{IEEE Transactions on Automatic Control}{}

\maketitle

\begin{abstract}
One of the modern research lines in econometrics studies focuses on translating a wide variety of structural econometric models into their state-space form, which allows for efficient unknown dynamic system state and parameter estimations by the Kalman filtering scheme. The mentioned trend yields advanced state-space model structures, which demand innovative estimation techniques driven by application requirements to be devised. This paper explores both the linear time-invariant multiple-input, multiple-output system (LTI MIMO) and the pairwise Markov model (PMM) with the related pairwise Kalman filter (PKF). In particular, we design robust gradient-based adaptive Kalman-like filtering methods for the simultaneous state and parameter estimation in the outlined model structures. Our methods are fast and accurate because their analytically computed gradient is utilized in the calculation instead of its numerical approximation. Also, these employ the numerically robust $UDU^\top$-factorization-based Kalman filter implementation, which is reliable in practice. Our novel techniques are examined on numerical examples and used for treating one stochastic model in econometrics.
\end{abstract}


\IEEEpeerreviewmaketitle


\section{Introduction}

The state-space approach using a Kalman filter (KF) has been extensively explored in the econometrics discipline in recent years. Although it took some time to conceptualize and adopt optimal filtering theory to econometric applications, its popularity has spread rapidly since the mid-1980s and the first contributions
promoted the KF methodology in the financial industry~\cite{1989:Harvey:book,2012:Durbin:book}.
An excellent and comprehensive survey of such state-space-approach-based applications in econometrics can be found in~\cite{2017:Wilcox}. Among the earliest implementations, we can distinguish the time-varying regression coefficient estimation problem in~\cite{1973:Sarris,1973:Rosenberg} and the time-varying variance (volatility) model calibration discussed in~\cite{1991:Hall:note,1991:Hall,1994:Ruiz,1994:Harvey,2009:Koopman,2006:Tims}. However, the modern econometric applications imply even more sophisticated  system structures stemmed from the modeling complicated `stylized' facts such as the non-stationarity of underlying stochastic processes discussed in~\cite{2006:Dahlhaus,2008:Engle}. In turn, all this requests that the advanced state-space model structures and their related estimation techniques~\cite{2007:Elliott,2013:Ossandon,2017:Ferreira} to be utilized as well as the innovative filtering methods dealing with nonlinearity and/or heavy-tailed uncertainties~\cite{1998:Fridman,2004:Jacquier,2014:Harvey}. In particular, this paper focuses on the first trend, i.e. on the design of advanced estimation methods for the innovative state-space model structures, which extend the classical state-space representation. More precisely, we investigate the linear time-invariant multiple-input, multiple-output (LTI MIMO) systems discussed in~\cite{2005:Gibson} as well as the pairwise Markov models (PMMs) with the associated pairwise KF (PKF) proposed in~\cite{2003:Pieczynski,2011:Lanchantin}.

The maximum likelihood estimation of econometrics models is often carried out by the Expectation-Maximization (EM) algorithm; e.g. see~\cite{1982:Shumway,1993:Koopman,2007:Elliott}, and many others. For the stochastic models under consideration, the cited estimation strategy has been
designed as well. For instance, such EM-based estimators are found for the LTI MIMO systems in~\cite{2005:Gibson} and for the linear PMMs in~\cite{2013:Nemesin}, respectively. An attractive feature of the EM-algorithm is its {\it derivative-free} nature, i.e. the likelihood function gradient (score) evaluation is not required there~\cite{1977:Dempster}. Among the EM-method's disadvantages, we may mention its slow convergence~\cite{2014:Mader} and difficulties in {\it partial} learning~\cite{2015:Nemesin}. The alternative {\it gradient-based} approach provides a higher convergence rate in comparison to that of the EM-algorithm strategy but demands additionally the score evaluation~\cite{1974:Gupta}. The likelihood function and its derivatives computation (with respect to the unknown system's parameters) can be fulfilled by using either the filtering equations, only, (see~\cite{1989:Zadrozny,2000:Klein}) or by implementing the united filtering and smoothing approach~\cite{1988:Segal,1992:Koopman,2008:Wills}. Below, we avoid any smoother implementation with its associated derivative computation but restrict ourselves to the first technology outlined.

Motivated by our results~\cite{2013:Tsyganova:IEEE,2017:Kulikova:IFAC}, we design robust gradient-based adaptive filtering methods for simultaneous state and parameter estimations in the model structures under consideration. Following the KF-based score evaluation technique in~\cite{1989:Zadrozny}, we present our filtering solution to the LTI MIMO and PMM state-space scenarios. However, the cited derivative evaluation scheme is grounded in differentiation of the conventional KF, which suffers from its numerical instabilities stemmed from the roundoff involved; see~\cite[Chapter~7]{2015:Grewal:book} for more details. In contrast, our novel methods proposed here are rooted in the numerically stable $UDU^\top$ factorization mechanization, which is reliable in practical calculations.
In addition, the use of analytically computed gradients instead of their numerical approximations in the optimization tool utilized reduces the execution time of the computational scheme designed. A similar conclusion has been recently reported for nonlinear systems in~\cite{2016:Boiroux}. Thus, the new methods derived are fast, numerically stable, and provide reliable computations in the econometric modeling. It is worth mentioning here that alternative factored-form  techniques do exist in the classical KF theory. These are based either on the Cholesky factorization or the singular value decomposition~\cite{2009:Kulikova:IEEE,2017:Tsyganova:IEEE}. Note that one Cholesky-based algorithm is recently developed in~\cite{2017:Kulikova:IEEE:PKF} in the realm of PKF, only. Meanwhile, no method has been still designed for treating the LTI MIMO systems, except those covered in~\cite{2008:Wills} where the study focuses on the optimization method itself rather than on its numerical robustness.

In fact, the derivation of novel filtering algorithms is valued for its own sake because it provides practitioners with a diversity of methods giving a fair capacity for choosing the most relevant one to a real-world application at hand, its complexity and accuracy requirements. For instance, the Cholesky gradient-based adaptive KF methodology is shown to succeed in the nonlinear biomedical signal processing~\cite{2018:Boureghda}. In contrast, our approach enjoys the square-root-free Cholesky-based factorization, which can be advantageous in econometrics estimation tasks. Here, we explore the time-varying GARCH-in-Mean(1,1) models and their estimation techniques, which are used for recovering the weak-form market efficiency, as that elaborated in~\cite{1999:Zalewska}. Its severe difficulty lies in the assumption of time-varying regression coefficients employed. In particular, our market efficiency study concerns with the EURO STOXX 50 index, which covers the largest and most liquid $50$ stocks in $18$ European countries.

\section{Problem statement and conventional methods} \label{problem:statement}

Consider a classical state-space model with the notation traditionally used in the econometrics discipline~\cite{2012:Durbin:book}, as follows:
  \begin{align}
     \alpha_{k+1}& =  T \alpha_{k} + \eta_k, &   \eta_k & \sim  {\cal N} (0, Q),  \label{eq:st:1} \\
     y_k & =  Z \alpha_k + \varepsilon_k, &  \varepsilon_k & \sim  {\cal N} (0, H)    \label{eq:st:2}
  \end{align}
where the subscript $k$ ($k \ge 0$) refers to the discrete time, i.e. $y_k$ means $y(t_k)$ and so on. The random vectors $\alpha_k \in \mathbb R^n$ and $y_k \in \mathbb R^m$ are the unknown (hidden) dynamic state to be estimated and available measurement vector, respectively.
The process and measurement noise covariances are symmetric, positive (semi-) definite matrices, i.e. $Q\ge 0$ and $H >0$. The uncertainties $\eta_k$ and $\varepsilon_k$ as well as the initial state vector $\alpha_0 \sim {\cal N}(\bar \alpha_0, \Pi_0)$ where $\Pi_0 > 0$ are assumed to be mutually uncorrelated Gaussian white noise processes.

Any filtering technique aims to recover the unknown (hidden) random sequence $\{\alpha_k\}_{k=0}^{N} = \{\alpha_0, \ldots, \alpha_N\}$  from the observed one $\{y_k\}_{k=0}^{N} = \{y_0, \ldots, y_N\}$. More formally, having sampled the measurements $\{ y_0, \ldots, y_k \}$, the classical KF is a recursive method for yielding the minimum-linear-expected-mean-square-error estimate $\hat \alpha_{k|k}$ of dynamic state~\eqref{eq:st:1} at time $t_k$; see~\cite[Theorem~9.2.1]{2000:Kailath:book}.

Many econometrics models often maintain a more general state-space form. The observation equation~\eqref{eq:st:2} typically includes an extra sequence $\{x_k\}_{k=0}^{N} = \{x_0, \ldots, x_N\}$ of explanatory variables; see the classical GARCH$(p,q)$ description in~\cite{1986:Bollerslev}, which  incorporates a term $x_k^\top\beta$ with the vector $\beta$ standing for the regression coefficients involved (usually, these are unknown and need to be estimated in parallel with the state). Thus, we arrive at the following extended form of measurement equation~\eqref{eq:st:2} in the above stochastic system:
\begin{align}
   y_k & =  Z \alpha_k + \beta x_k + \varepsilon_k, &  \varepsilon_k & \sim  {\cal N} (0, H)    \label{eq:st:2:e}
\end{align}
where the vector $x_k \in {\mathbb R}^d$ refers to the explanatory variables, which is also interpreted as a controllable input in engineering applications~\cite{2017:Wilcox}. Eventually, the conventional state-space representation~\eqref{eq:st:1} and \eqref{eq:st:2} is extended to the following more general LTI MIMO form~\cite{2005:Gibson}:
  \begin{align}
  \label{eq:lti:model}
  \begin{bmatrix}
   \alpha_{k+1} \\
   y_k
   \end{bmatrix}  = &
  \begin{bmatrix}
   T &  \!\!B \\
   Z &  \!\!\beta
   \end{bmatrix}
     \begin{bmatrix}
   \alpha_{k} \\
   x_{k}
   \end{bmatrix} +
        \begin{bmatrix}
   \eta_{k} \\
   \varepsilon_{k}
   \end{bmatrix}\!,
        \begin{bmatrix}
   \eta_{k} \\
   \varepsilon_{k}
   \end{bmatrix}
\!\sim \!{\cal N}\!\!\left(
\begin{bmatrix}
0 \\ 0
\end{bmatrix}\!\!,
\begin{bmatrix}
Q & \!\!S  \\
S^\top & \!\!H
\end{bmatrix}
\right)
\end{align}
where the cross-correlation between the process and measurement uncertainties is additionally assumed. Note that the substitution of $\beta=0$, $B=0$ and $S=0$ into~\eqref{eq:lti:model} reduces it to the usual state-space stochastic system~\eqref{eq:st:1} and \eqref{eq:st:2}. The KF-like estimator for the LTI MIMO model~\eqref{eq:lti:model} is commonly presented as follows~\cite[Lemma~3.2]{2005:Gibson}:
\begin{codebox}
\Procname{{\bf Algorithm~1}. $\proc{LTI MIMO KF}$ ({\it conventional algorithm})}
\zi \textsc{Initialization:} ($k=0$)
\li Set $\hat \alpha_{0|0} = \bar \alpha_0$ and $P_{0|0} = \Pi_0$;
\li $\overline{T}  = T - SH^{-1}Z$, $\overline{B}  = B - SH^{-1}\beta$, $\overline{Q}  = Q - SH^{-1}S^\top$;
\zi \textsc{Time Update}: ($k=\overline{1,N}$) \Comment{\small\textsc{Priori estimation}}
\li \>$\hat \alpha_{k|k-1}  = \overline{T} \hat \alpha_{k-1|k-1} + \overline{B}x_{k-1}+SH^{-1}y_{k-1}$; \label{lti:p:X}
\li \>$P_{k|k-1}  = \overline{T} P_{k-1|k-1}\overline{T}^\top + \overline{Q}$; \label{lti:p:P}
\zi \textsc{Measurement Update}: \Comment{\small\textsc{Posteriori estimation}}
\li \>$K_{k}  = P_{k|k-1}Z^\top R_{e,k}^{-1}$ where $R_{e,k} = ZP_{k|k-1}Z^\top+H$; \label{lti:f:K}
\li \>$\hat \alpha_{k|k}  =  \hat \alpha_{k|k-1}+K_{k}e_k$ where $e_k = y_k-Z\hat \alpha_{k|k-1}-\beta x_k$;  \label{lti:f:X}
\li \>$P_{k|k}  = (I - K_{k}Z)P_{k|k-1}$. \label{lti:f:P}
\end{codebox}

We remark that the described LTI MIMO system's structure can be even more sophisticated. In particular, the linear Gaussian PMM under study follows from the state-space form elaborated in~\cite{2013:Nemesin}, i.e.
  \begin{align}
  \label{eq:pkf:model}
  \begin{bmatrix}
   \alpha_{k+1} \\
   y_k
   \end{bmatrix} \!\! = &
  \begin{bmatrix}
   T &  \!\!B \\
   Z &  \!\!\beta
   \end{bmatrix}
     \begin{bmatrix}
   \alpha_{k} \\
   y_{k-1}
   \end{bmatrix}\!\! +
        \begin{bmatrix}
   \eta_{k} \\
   \varepsilon_{k}
   \end{bmatrix}\!\!,
        \begin{bmatrix}
   \eta_{k} \\
   \varepsilon_{k}
   \end{bmatrix}
\!\sim \!{\cal N}\!\!\left(
\begin{bmatrix}
0 \\ 0
\end{bmatrix}\!\!,
\begin{bmatrix}
Q & \!\!S  \\
S^\top & \!\!H
\end{bmatrix}
\right)
\end{align}
where the pair of two random sequences $(\alpha_k,y_k)$ is assumed to be Markovian; see further details in~\cite{2011:Lanchantin}.
While the PKF is first derived in~\cite[Proposition 1]{2003:Pieczynski}, its summary is available in~\cite{2013:Nemesin} as follows:
\begin{codebox}
\Procname{{\bf Algorithm~2}. $\proc{PKF}$ ({\it conventional implementation})}
\zi \textsc{Initialization:} ($k=0$)
\li Set $\hat \alpha_{0|0} = \bar \alpha_0$ and $P_{0|0} = \Pi_0$;
\li $\overline{T}  = T - SH^{-1}Z$, $\overline{B}  = B - SH^{-1}\beta$, $\overline{Q}  = Q - SH^{-1}S^\top$;
\zi \textsc{Time Update}: ($k=\overline{0,N-1}$) \Comment{\small\textsc{Priori estimation}}
\li \>$\hat \alpha_{k+1|k}  = \overline{T} \hat \alpha_{k|k} + \overline{B}y_{k-1}+SH^{-1}y_{k}$; \label{pkf:p:X}
\li \>$P_{k+1|k}  = \overline{T} P_{k|k}\overline{T}^\top + \overline{Q}$; \label{pkf:p:P}
\zi \textsc{Measurement Update}: \Comment{\small\textsc{Posteriori estimation}}
\li \>$K_{k+1}  = P_{k+1|k}Z^\top R_{e,k+1}^{-1}$ with $R_{e,k+1} \!= ZP_{k+1|k}Z^\top\!+H$; \label{pkf:f:K}
\li \>$e_{k+1} = y_{k+1}-Z\hat \alpha_{k+1|k}-\beta y_{k}$; \label{pkf:f:e}
\li \>$\hat \alpha_{k+1|k+1}  =  \hat \alpha_{k+1|k}+K_{k+1}e_{k+1}$; \label{pkf:f:X}
\li \>$P_{k+1|k+1}  = (I - K_{k+1}Z)P_{k+1|k}$. \label{pkf:f:P}
\end{codebox}

Note that the econometrics models are usually parameterized, i.e. in addition to the filtering task, one needs to pay attention to a parameter estimation one as well. For example, the regression coefficients $\beta$ are typically unknown and should be estimated together with the dynamic state $\alpha_k$ from the available data $Y_N:=\{ y_k \}_{k=0}^N$. In general, entries of the system's matrices $\{T, Z, B, \beta, Q, H\}$ may depend on unknown parameters $\theta \in {\mathbb R}^p$, which are also to be estimated. The maximum likelihood approach is a common tool for parameter estimations in econometrics, where the underlying filter allows the following log likelihood function (LF) to be evaluated, as explained in~\cite{1965:Schweppe,1989:Harvey:book}:
\begin{align}
\ln L\left(\theta | Y_N\right)  = - c_0 - \frac{1}{2}\sum \limits_{k=0}^N \left\{\ln\left(\det R_{e,k}\right) + e_k^\top R_{e,k}^{-1}e_k \right\} \label{eq:llf}
\end{align}
with $c_0$ being a constant. The derivation of~\eqref{eq:llf} relies on the optimality of the KF in the linear Gaussian state-space models where the filters' innovations obey $e_k \sim {\cal N}\left(0, R_{e,k}\right)$. In non-Gaussian settings, the KF exhibits a sub-optimal behaviour, only, and the mentioned property is violated. Therefore, equation~\eqref{eq:llf} yields a quasi-maximum likelihood estimation in such non-Gaussian settings, as discussed in~\cite{1994:Ruiz,1994:Harvey}.

Eventually, any parameter estimation mechanization includes a filter utilized for computing the log LF~\eqref{eq:llf} and an optimization routine for determining the optimal parameter values $\theta^{*}$. All this sets up an {\it adaptive} filtering framework because it turns out that the filter learns itself based on the information available from the measurements.

The gradient-based optimization methods converge faster. So, these are preferable ones in practice~\cite{1974:Gupta}. Besides, such methods
demand the log LF gradient evaluation (which is called the {\it score} in statistics) and might require the Fisher information matrix computation as well. It is readily seen from formula~\eqref{eq:llf} that the score calculation yields the problem of the filter differentiation, which results in the vector- and matrix-type equations known as the {\it filter sensitivity} equations. An obvious solution for deriving the sensitivity model is to differentiate the underlying filtering equations with respect to each $\theta_i$, $i=1, \ldots, p$. This standard approach is implemented in~\cite{1989:Zadrozny,2000:Klein} and~\cite{2008:Wills} for the classical state-space models and LTI MIMO systems, respectively. In the PKF, the related sensitivity equations have never been presented explicitly. However, it is worth noting that these are similar to those in~\cite{1989:Zadrozny,2000:Klein}. It follows from the fact that the observations $y_k$ and the explanatory variables (control input in engineering applications) $x_k$  do not depend on the unknown system's parameters (i.e. their realizations are independent of variations in $\theta$) and, hence, their differentials equal to zero. Although it might seem that the problem has been solved, the truth is that this traditional approach is actually unreliable because of numerical instabilities arisen in the conventional KF, which affect their derivative calculations as well. The roundoff errors committed might force the filter to fail and, then, demolish performance of the entire adaptive KF scheme applied in various engineering tasks. This issue has been already addressed in our previous studies~\cite{2009:Kulikova:IEEE,2013:Tsyganova:IEEE,2017:Tsyganova:IEEE} by using the robust factored-form (square-root) KF algorithms and their derivatives. The cited techniques are shown to be reliable and numerically stable. That is why these are promising tools for practical use. Further, we follow the modified Cholesky-based KF methodology designed in~\cite{2013:Tsyganova:IEEE} and extend it to the LTI MIMO and PMM models under consideration.

\section{Main results} \label{main:result}
\subsection{The UDU$^\top$-factorization-based filtering}

The $UDU^\top$-factorization-based approach to implementing the KF uses the modified Cholesky decomposition ${P=\bar U_{P}D_{P}\bar U^\top_{P}}$ where $D_{P}$ denotes a diagonal matrix and $\bar U_{P}$ stands for an upper (or lower) triangular matrix with the unitary main diagonal. This factorization is performed in the filter's initialization step. Then, the standard filtering equations are to be mathematically re-arranged for propagating and updating these factors rather than the entire covariance matrix $P$ in the underlying Riccati recursion. The outlined methodology ensures the symmetric form of the error covariance and its positive semi-definiteness at any time instance $t_k$. Note that this covariance is trivially restored from its factors as follows: $P:=\bar U_{P}D_{P}\bar U^\top_{P}$. We term such algorithms as {\it UD-based} KFs. In what follows, the operation $A \oplus B$ with square matrices $A$ and $B$ of sizes $m$ and $n$ results in the block-diagonal matrix ${\rm diag}\{A,B\}$ of size~$(m+n)\times(m+n)$.

The first UD-based KFs were designed in~\cite{1976:Thornton:PhD,1977:Bierman}. These were called the square-root-free ones and shown to have the numerical stability similar to that of the Cholesky-based algorithms (i.e. the UD-based are also more reliable in comparison to the conventional KFs), but faster because of avoiding the calculation of square roots. Note that a recursive propagation of the factors $\bar U_p$ and $D_p$ can be fulfilled by orthogonal transforms, which set up conformal (i.e. a norm- and angle-preserving) mappings. So, the estimation quality and robustness of such implementations have been greatly improved compared to the first UD-based KFs presented in~\cite{1976:Thornton:PhD,1977:Bierman}. Nowadays, the UD-based KFs  always use some sort of weighted orthogonal transforms. Here, we follow the approach grounded in the modified weighted Gram-Schmidt (MWGS) orthogonalization~\cite[Lemma~VI.4.1]{1977:Bierman:book}: given pre-arrays $\{ {\mathbb A}, {\mathbb D}_{A} \}$ with ${\mathbb A} \in {\mathbb R}^{r\times s}$, $r \ge s$, and the diagonal matrix ${\mathbb D}_{A} \in {\mathbb R}^{r\times r}>0$, compute the post-arrays $ \{ {\mathbb R}, {\mathbb D}_{R} \}$ with the block upper triangular matrix ${\mathbb R} \in {\mathbb R}^{s\times s}$ having the unitary diagonal and the diagonal one ${\mathbb D}_{R} \in {\mathbb R}^{s\times s}$ by the MWGS method, which satisfy
\begin{equation}
\label{assume:1}
{\mathbb A}^\top={\mathbb R}\:\mathfrak{W}^\top \quad \mbox{ where } \quad \mathfrak{W}^\top{\mathbb D}_{A}\mathfrak{W}={\mathbb D}_{R}
\end{equation}
where $\mathfrak{W}  \in {\mathbb R}^{r\times s}$ stands for the matrix of the MWGS transform.

The above consideration yields the following UD-based KF:
\begin{codebox}
\Procname{{\bf Algorithm~1a}. $\proc{UD-based LTI MIMO KF}$ ({\it UD-based algorithm})}
\zi \textsc{Initialization:} ($k=0$)
\li Apply the factorization: $\Pi_0 = \bar U_{\Pi_0}D_{\Pi_0}\bar U_{\Pi_0}^\top$, $H = \bar U_{H}D_{H}\bar U_{H}^\top$;
\li Set $\hat \alpha_{0|0} = \bar \alpha_0$ and $\bar U_{P_{0|0}} = \bar U_{\Pi_0}$, $D_{P_{0|0}} = D_{\Pi_0}$;
\li $\overline{T}  = T - SH^{-1}Z$, $\overline{B}  = B - SH^{-1}\beta$, $\overline{Q}  = Q - SH^{-1}S^\top$;
\li Apply the $UDU^\top$ factorization: $\overline{Q} = \bar U_{\overline{Q}}D_{\overline{Q}}\bar U_{\overline{Q}}^\top$;
\zi \textsc{Time Update}: ($k=\overline{1,N}$) \Comment{\small\textsc{Priori estimation}}
\li \>$\hat \alpha_{k|k-1}  = \overline{T} \hat \alpha_{k-1|k-1} + \overline{B}x_{k-1}+SH^{-1}y_{k-1}$; \label{lti:ud:p:X}
\li \>Build pre-arrays ${\mathbb A}$, ${\mathbb D}_A$ and apply the MWGS algorithm \label{lti:ud:p:P}
\zi \>$\phantom{\mathfrak{V}^\top}\underbrace{
\begin{bmatrix}
\overline{T} \bar U_{P_{k-1|k-1}} & \; \bar U_{\overline{Q}}
\end{bmatrix}
}_{\mbox{\scriptsize Pre-array ${\mathbb A}^\top$}}
 =
  \underbrace{
\begin{bmatrix}
\bar U_{P_{k|k-1}}
\end{bmatrix}
}_{\mbox{\scriptsize Post-array ${\mathbb R}$}}{\mathfrak V}^\top$;
\zi \>$\mathfrak{V}^\top
\underbrace{
\left[D_{P_{k-1|k-1}}\oplus D_{\overline{Q}}\right]
}_{\mbox{\scriptsize Pre-array ${\mathbb D}_{A}$}}
\mathfrak{V}
=  \underbrace{
\begin{bmatrix}
D_{P_{k|k-1}}
\end{bmatrix}
}_{\mbox{\scriptsize Post-array ${\mathbb D}_{R}$}}$;
\zi \>Read-off from the post-arrays: $\bar U_{P_{k|k-1}}$ and $D_{P_{k|k-1}}$;
\end{codebox}
\begin{codebox}
\setlinenumberplus{lti:ud:p:P}{1}
\zi \textsc{Measurement Update}: \Comment{\small\textsc{Posteriori estimation}}
\li \>Build pre-arrays ${\mathbb A}$, ${\mathbb D}_A$ and apply the MWGS algorithm \label{lti:ud:f:P}
\zi \>$\phantom{\mathfrak{Q}^\top} \underbrace{
\begin{bmatrix}
\bar U_{P_{k|k-1}} & 0\\
Z\bar U_{P_{k|k-1}} & \bar U_{H}
\end{bmatrix}
}_{\mbox{\scriptsize Pre-array ${\mathbb A}^\top$}}
 =
  \underbrace{
\begin{bmatrix}
\bar U_{P_{k|k}} & \bar K_{k}^u\\
0 & \bar U_{R_{e,k}}
\end{bmatrix}
}_{\mbox{\scriptsize Post-array ${\mathbb R}$}}{\mathfrak Q}^\top$;
\zi\> $\mathfrak{Q}^\top
\underbrace{
\left[D_{P_{k|k-1}}\oplus D_{H}\right]
}_{\mbox{\scriptsize Pre-array ${\mathbb D}_{A}$}}
\mathfrak{Q}
= \underbrace{
\left[D_{P_{k|k}}\oplus D_{R_{e,k}}\right]
}_{\mbox{\scriptsize Post-array ${\mathbb D}_{R}$}}$;
\zi \>Read-off the factors: $\bar U_{R_{e,k}}$, $D_{R_{e,k}}$, $\bar U_{P_{k|k}}$, $D_{P_{k|k}}$, $\bar K_{k}^u$;
\li \>$e_k = y_k-Z\hat \alpha_{k|k-1}-\beta x_k$; \label{lti:ud:f:ek}
\li \>$\hat \alpha_{k|k}  =  \hat \alpha_{k|k-1}+ [\bar K_{k}^u][\bar U_{R_{e,k}}]^{-1}e_k$. \label{lti:ud:f:X}
\end{codebox}

\begin{Prop} \label{proposition:1}
The UD-based KF equations in Algorithm~1a are algebraically equivalent to the KF formulas underlying Algorithm~1.
\end{Prop}

\begin{proof} The orthogonal transform used in equations~\eqref{assume:1} sets up a conformal mapping between the (block) rows of the pre-array ${\mathbb A}^\top$ and the rows of the post-array ${\mathbb R}$, which preserves all diagonally weighted norms. The weighted inner products underlying the algebraic manipulations shown in lines~\ref{lti:ud:p:P},~\ref{lti:ud:f:P} of Algorithm~1a yield
\[
\Bigl([\overline{T} \bar U_{P_{k-1|k-1}}, \bar U_{\overline{Q}}], [\overline{T} \bar U_{P_{k-1|k-1}}, \bar U_{\overline{Q}}]\Bigr)_{{\mathbb D}_{A}} = \Bigl([N, 0],[N, 0]\Bigr)_{D_0}
\]
where the pre-array ${\mathbb D}_{A} = {D_{P_{k-1|k-1}}\!\oplus\,D_{\overline{Q}}}$. Thus, we obtain
\[ND_0N^\top = \overline{T}\underbrace{(\bar U_{P_{k-1|k-1}}D_{P_{k-1|k-1}}\bar U_{P_{k-1|k-1}}^\top)}_{P_{k-1|k-1}}\overline{T}^\top  +
\overline{Q}.\]
A comparison of this equation to the formula in line~\ref{lti:p:P} of Algorithm~1 says that $N=\bar U_{P_{k|k-1}}$ and $D_0=D_{P_{k|k-1}}$. Then, we consider
\begin{align}
&\Bigl([\bar U_{P_{k|k-1}},  0], [\bar U_{P_{k|k-1}},  0]\Bigr)_{{\mathbb D}_{A}} = \Bigl([X, Y], [X, Y]\Bigr)_{D_1\oplus D_2}\!, \label{proof:2} \\
&\Bigl([\bar U_{P_{k|k-1}}, 0], [Z\bar U_{P_{k|k-1}}, \bar U_{H}]\Bigr)_{{\mathbb D}_{A}}\!\!= \Bigl([X, Y], [0, M]\Bigr)_{D_1\oplus D_2}\!, \label{proof:3} \\
&\Bigl([Z\bar U_{P_{k|k-1}}, \bar U_{H}], [Z\bar U_{P_{k|k-1}}, \bar U_{H}]\Bigr)_{{\mathbb D}_{A}} \label{proof:4} \\
& = \Bigl([0, M], [0,  M]\Bigr)_{D_1\oplus D_2} \mbox{ where } {\mathbb D}_{A} = D_{P_{k|k-1}}\!\oplus\,D_{H}.
\end{align}

Formula~\eqref{proof:4} implies the evident equality
\[MD_2M^\top = Z\underbrace{\bar U_{P_{k|k-1}} D_{P_{k|k-1}} \bar U_{P_{k|k-1}}^\top}_{P_{k|k-1}}Z^\top+\underbrace{\bar U_H D_H \bar U_H^\top}_{H}.\]
This and line~\ref{lti:f:K} in Algorithm~1 state that $M=\bar U_{R_{e,k}}$, $D_2= D_{R_{e,k}}$.

Next, the substitution of the matrices $M$ and $D_2$ into~\eqref{proof:3} yields
\[Y D_{R_{e,k}}\bar U_{R_{e,k}}^\top = (\bar U_{P_{k|k-1}} D_{P_{k|k-1}} \bar U_{P_{k|k-1}}^\top)Z^\top=P_{k|k-1}Z^\top.\]
It is further solved for $Y$, which leads to $Y = P_{k|k-1}Z^\top\bar U_{R_{e,k}}^{-T}D_{R_{e,k}}^{-1}$. In what follows, this value is denoted as $Y:=\bar K_k^{u}$, which is the so-called  ``normalized'' gain in the UD-based LTI MIMO KFs. Following line~\ref{lti:f:K} in Algorithm~1, we establish an obvious relationship $K_k = P_{k|k-1}Z^\top R_{e,k}^{-1} = \bar K_k^{u} \bar U_{R_{e,k}}^{-1}$. Eventually, the ``normalized'' gain is directly read-off from the post-array and, then, the equation for computing the state estimate (see line~\ref{lti:f:X} in Algorithm~1) is re-formulated in terms of the available post-array's blocks $\bar K_k^{u}$ and $\bar U_{R_{e,k}}$, and the filter's innovations $e_k$, as follows:
\[ \hat \alpha_{k|k}  =   \hat \alpha_{k|k-1}+K_{k}e_k =  \hat \alpha_{k|k-1}+ [\bar K_k^{u}] [\bar U_{R_{e,k}}]^{-1} e_k.   \]
Thus, the equation in line~\ref{lti:ud:f:X} of Algorithm~1a has been validated.

Finally, the substitution of the above-derived value $Y$ into~\eqref{proof:2} yields
\[
  P_{k|k-1}  =  \bar U_{P_{k|k-1}} D_{P_{k|k-1}} \bar U_{P_{k|k-1}}^\top = XD_1X^\top  + \bar K_k^{u} D_{R_{e,k}} (\bar K_k^{u})^\top
  \]
and the equality $K_k = P_{k|k-1}Z^\top R_{e,k}^{-1} = \bar K_k^{u} \bar U_{R_{e,k}}^{-1}$ then establishes
\begin{align*}
 XD_1X^\top & = P_{k|k-1} - \bar K_k^{u} D_{R_{e,k}} (\bar K_k^{u})^\top \\
 & = P_{k|k-1}-P_{k|k-1}Z^\top\bar U_{R_{e,k}}^{-T}D_{R_{e,k}}^{-1}D_{R_{e,k}}(K_k\bar U_{R_{e,k}})^\top \\
 & = P_{k|k-1} - P_{k|k-1}Z^\top K_k^\top = P_{k|k-1}(I - Z^\top K_k^\top),
\end{align*}
which is precisely the formula in line~\ref{lti:f:P} of Algorithm~1, because the error covariance matrix is symmetric by definition. We conclude that $X=\bar U_{P_{k|k}}$ and $D_1= D_{P_{k|k}}$. This completes the proof.
\end{proof}

Similarly, the UD-based PKF has been recently presented in~\cite{2017:Kulikova:IFAC}. Here, we briefly outline it for readers' convenience, as follows:

\begin{codebox}
\Procname{{\bf Algorithm~2a}. $\proc{UD-based PKF}$ ({\it UD-based algorithm})}
\zi \textsc{Initialization:} ($k=0$)
\li Apply the factorization: $\Pi_0 = \bar U_{\Pi_0}D_{\Pi_0}\bar U_{\Pi_0}^\top$, $H = \bar U_{H}D_{H}\bar U_{H}^\top$;
\li Set $\hat \alpha_{0|0} = \bar \alpha_0$ and $\bar U_{P_{0|0}} = \bar U_{\Pi_0}$, $D_{P_{0|0}} = D_{\Pi_0}$;
\li $\overline{T}  = T - SH^{-1}Z$, $\overline{B}  = B - SH^{-1}\beta$, $\overline{Q}  = Q - SH^{-1}S^\top$;
\li Apply the $UDU^\top$ factorization: $\overline{Q} = \bar U_{\overline{Q}}D_{\overline{Q}}\bar U_{\overline{Q}}^\top$;
\zi \textsc{Time Update}: ($k=\overline{0,N-1}$) \Comment{\small\textsc{Priori estimation}}
\li \>$\hat \alpha_{k+1|k}  = \overline{T} \hat \alpha_{k|k} + \overline{B}y_{k-1}+SH^{-1}y_{k}$; \label{pkf:ud:p:X}
\li \>Build pre-arrays ${\mathbb A}$, ${\mathbb D}_A$ and apply the MWGS algorithm \label{pkf:ud:p:P}
\zi \>$\phantom{\mathfrak{V}^\top}\underbrace{
\begin{bmatrix}
\overline{T} \bar U_{P_{k|k}} & \; \bar U_{\overline{Q}}
\end{bmatrix}
}_{\mbox{\scriptsize Pre-array ${\mathbb A}^\top$}}
 =
  \underbrace{
\begin{bmatrix}
\bar U_{P_{k+1|k}}
\end{bmatrix}
}_{\mbox{\scriptsize Post-array ${\mathbb R}$}}{\mathfrak V}^\top
$;
\zi \>$\mathfrak{V}^\top
\underbrace{
\left[D_{P_{k|k}}\oplus D_{\overline{Q}}\right]
}_{\mbox{\scriptsize Pre-array ${\mathbb D}_{A}$}}
\mathfrak{V}
=  \underbrace{
\begin{bmatrix}
D_{P_{k+1|k}}
\end{bmatrix}
}_{\mbox{\scriptsize Post-array ${\mathbb D}_{R}$}}$;
\zi \>Read-off from the post-arrays: $\bar U_{P_{k+1|k}}$ and $D_{P_{k+1|k}}$;
\zi \textsc{Measurement Update}: \Comment{\small\textsc{Posteriori estimation}}
\li \>Build pre-arrays ${\mathbb A}$, ${\mathbb D}_A$ and apply the MWGS algorithm \label{pkf:ud:f:P}
\zi \>$\phantom{\mathfrak{Q}^\top} \underbrace{
\begin{bmatrix}
\bar U_{P_{k+1|k}} & 0\\
Z\bar U_{P_{k+1|k}} & \bar U_{H}
\end{bmatrix}
}_{\mbox{\scriptsize Pre-array ${\mathbb A}^\top$}}
 =
  \underbrace{
\begin{bmatrix}
\bar U_{P_{k+1|k+1}} & \bar K_{k+1}^u\\
0 & \bar U_{R_{e,k+1}}
\end{bmatrix}
}_{\mbox{\scriptsize Post-array ${\mathbb R}$}}{\mathfrak Q}^\top$;
\zi\> $\mathfrak{Q}^\top
\underbrace{
\left[D_{P_{k+1|k}}\oplus D_{H}\right]
}_{\mbox{\scriptsize Pre-array ${\mathbb D}_{A}$}}
\mathfrak{Q}
= \underbrace{
\left[D_{P_{k+1|k+1}}\oplus D_{R_{e,k+1}}\right]
}_{\mbox{\scriptsize Post-array ${\mathbb D}_{R}$}}$;
\zi \>Get $\bar U_{R_{e,k+1}}$, $D_{R_{e,k+1}}$, $\bar U_{P_{k+1|k+1}}$, $D_{P_{k+1|k+1}}$, $\bar K_{k+1}^u$;
\li \>$e_{k+1} = y_{k+1}-Z\hat \alpha_{k+1|k}-\beta y_k$; \label{pkf:ud:f:ek}
\li \>$\hat \alpha_{k+1|k+1}  =  \hat \alpha_{k+1|k}+[\bar K_{k+1}^u][\bar U_{R_{e,k+1}}]^{-1}e_{k+1}$. \label{pkf:ud:f:X}
\end{codebox}

\subsection{The UDU$^\top$-based filters' derivative computation}

The gradient-based adaptive filtering schemes require the log LF gradient evaluation (with respect to unknown system's parameters). We recall that the log LF in formula~\eqref{eq:llf} can be expressed in terms of the UD-based filters' quantities under consideration, as follows~\cite{2013:Tsyganova:IEEE}:
\[
 \ln L\left(\theta | Y_N\right) =
- c_0 - \frac{1}{2} \sum \limits_{k=1}^N \left\{
  \ln\left(\det D_{R_{e,k}}\right)+ \bar e_k^\top D^{-1}_{R_{e,k}}\bar e_k
\right\}
\]
where the vector $\bar e_{k}=\bar U_{R_{e,k}}^{-1}e_{k}$ stands for the so-called ``normalized'' innovations of the corresponding UD-based KF implementation.

We denote the partial derivative of $(m \times n)$-matrix $A$ with respect to the $i$th entry in the parameter vector $\theta \in {\mathbb R}^p$ as $\Di{A} = \partial A/ \partial \theta_i$, $i=1,\ldots,p$. Then, the score admits the following representation in terms of the $UDU^\top$ factorization of the covariance matrix $R_{e,k}$~\cite{2013:Tsyganova:IEEE}:
\begin{align*}
\Di{\ln L\left(\theta | Y_N\right)} & =
-\frac{1}{2}\sum \limits_{k=1}^N \left\{
\tr { \DDi{D_{R_{e,k}}} D^{-1}_{R_{e,k}}}
 \right.\\
 & \left. - 2\DDi{\bar e_k^\top}D^{-1}_{R_{e,k}}\bar e_k -\bar e_k^\top D^{-2}_{R_{e,k}}\DDi{D_{R_{e,k}}}\bar e_k
\right\}.
\end{align*}
This formula requires the partial derivatives $\{ \Di{\bar e_k} \}$ and $\{ \Di{D_{R_{e,k}}} \}$ of the UD-based filters' variables to be available. For computing these quantities, one needs to express the {\it filter sensitivity} equations in terms of the UD-based KF under consideration. Thus, our next goal is to derive the sensitivity models for the LTI MIMO estimator (Algorithm~1a) and PKF (Algorithm~2a) that are mathematically equivalent to the standard approach based on the direct differentiation of the KF equations in Algorithms~1 and~2, respectively.
The theoretical result presented in~\cite[Lemma~1]{2013:Tsyganova:IEEE} suggests the computational technique, which can be summarized in the form of the following pseudocode:

\begin{codebox}
\Procname{{\bf Differentiated UD scheme}: $\proc{Diff. UD}({\mathbb A}, {\mathbb D}_A, {\mathbb A}'_{\theta}, ({\mathbb D}_A)'_{\theta})$}
\zi {\bf Input:} ${\mathbb A}$, ${\mathbb D}_A$, ${\mathbb A}'_{\theta}$, $({\mathbb D}_A)'_{\theta}$. \quad \Comment{\it \small Pre-arrays and its derivatives}
\li Apply~\eqref{assume:1} to the pre-arrays ${\mathbb A}$, ${\mathbb D}_A$. Get and save $\mathfrak{W}$, ${\mathbb R}$, ${\mathbb D}_R$;
\li Compute the matrix product $M_0 = \mathfrak{W}^\top{\mathbb D}_A{\mathbb A}'_\theta {\mathbb R}^{-\top}$;
\li $M_0=\bar L_0 + D_0 + \bar U_0$; \Comment{\it \small  Split into strictly lower triangular, }
\zi \phantom{$M_0=\bar L_0 + D_0 + \bar U_0$;} \Comment{\it \small diagonal, strictly upper triangular parts}
\li Compute the matrix product $M_2 = \mathfrak{W}^\top({\mathbb D}_A)'_\theta \mathfrak{W}$;
\li $M_2=\bar L_2 + D_2 + \bar U_2$; \Comment{\it \small  Split into strictly lower triangular, }
\zi \phantom{$M_2=\bar L_2 + D_2 + \bar U_2$;} \Comment{\it \small diagonal, strictly upper triangular parts}
\li Given $\bar L_0$, $\bar U_0$, $\bar U_2$, ${\mathbb R}$, ${\mathbb D}_R$, find ${\mathbb R}'_{\theta} = {\mathbb R}\left({\bar L}^\top_0+{\bar U}_0+{\bar U}_2\right){\mathbb D}^{-1}_R$;
\li Given $D_0$ and $D_2$, compute the derivative $\left({\mathbb D}_R\right)'_\theta=2D_0+D_2$;
\zi {\bf Output:} ${\mathbb R}$, ${\mathbb D}_R$ and $({\mathbb R})'_{\theta}$, $({\mathbb D}_R)'_{\theta}$. \Comment{\it \small Post-arrays and  derivatives}
\end{codebox}

Having accommodated this differentiated UD-based KF to the LTI MIMO estimator in Algorithm~1a, we arrive at its extended version, which allows the filter's sensitivities to be evaluated as follows:
\begin{codebox}
\Procname{{\bf Algorithm~1b}. $\proc{Differentiated UD-based LTI MIMO KF}$}
\zi \textsc{Initialization:} ($k=0$, $i=\overline{1,p}$)
\li \; Repeat initialization of Algorithm~1a in lines~1--4;
\li \; Compute $\Di{\overline{T}}$, $\Di{\overline{B}}$, $\Di{\overline{Q}}$, $\Di{\bar U_{\Pi_0}}$, $\Di{D_{\Pi_0}}$, $\Di{\bar \alpha_0}$;
\label{diff:lti:ud:initial1}
\li \; Set $\Di{\bar U_{P_{0|0}}}\!\! = \Di{\bar U_{\Pi_0}}$, $\Di{D_{P_{0|0}}}\!\! = \Di{D_{\Pi_0}}$, $\Di{\hat \alpha_{0|0}} = \Di{\bar \alpha_0}$; \label{diff:lti:ud:initial2}
\zi \textsc{Time Update}: ($k=\overline{1,N}$, $i=\overline{1,p}$) \Comment{\small\textsc{Priori estimation}}
\li \; Find {\it a priori} estimate $\hat \alpha_{k|k-1}$ from line~\ref{lti:ud:p:X} of Algorithm~1a;
\li \; Compute $\Di{\hat \alpha_{k|k-1}} = \DDi{\overline{T}}\hat \alpha_{k-1|k-1} + \overline{T}\DDi{\hat \alpha_{k-1|k-1}}$
\zi \; $+\DDi{\overline{B}} x_{k-1} + \DDi{S }H^{-1}y_{k-1} - S H^{-1}\DDi{H}H^{-1}y_{k-1}$; \label{diff:lti:p:X}
\li \; Build the pre-arrays ${\mathbb A}$, ${\mathbb D}_A$ in line~\ref{lti:ud:p:P} of Algorithm~1a; \label{diff:p:pre:array1}
\li \; Find the pre-arrays' derivatives $\Di{\mathbb A}$ and $\Di{{\mathbb D}_A}$;
\li \; Get $\left[{\mathbb R}, {\mathbb D}_R, \Di{{\mathbb R}}, \Di{{\mathbb D}_R} \right]\leftarrow$\verb"Diff. UD"(${\mathbb A}$, ${\mathbb D}_A$, $\Di{\mathbb A}$, $\Di{{\mathbb D}_A}$);
\li \; $\left\{ D_{P_{k|k-1}}, \: \Di{D_{P_{k|k-1}}} \right\} \leftarrow$ read-off from the ${\mathbb D}_R$ and $\Di{{\mathbb D}_R}$;
\li \; $\left\{\bar U_{P_{k|k-1}}, \: \Di{\bar U_{P_{k|k-1}}} \right\}\leftarrow$ read-off from the ${\mathbb R}$ and $\Di{{\mathbb R}}$; \label{diff:p:pre:array2}
\zi \textsc{Measurement Update}: \Comment{\small\textsc{Posteriori estimation}}
\li \; Build the pre-arrays ${\mathbb A}$, ${\mathbb D}_A$  in line~\ref{lti:ud:f:P} of Algorithm~1a; \label{diff:f:pre:array1}
\li \; Find the pre-arrays' derivatives $\Di{\mathbb A}$ and $\Di{{\mathbb D}_A}$;
\li \; Get $\left[{\mathbb R}, {\mathbb D}_R, \Di{{\mathbb R}}, \Di{{\mathbb D}_R} \right]\leftarrow$\verb"Diff. UD"(${\mathbb A}$, ${\mathbb D}_A$, $\Di{\mathbb A}$, $\Di{{\mathbb D}_A}$);
\li \; $\left\{ D_{R_{e,k}}, \Di{D_{R_{e,k}}} \right\}\leftarrow$ read-off from the ${\mathbb D}_R$ and $\Di{{\mathbb D}_R}$;
\li \; $\left\{ D_{P_{k|k}}, \Di{D_{P_{k|k}}} \right\}\leftarrow$ read-off from the ${\mathbb D}_R$ and $\Di{{\mathbb D}_R}$;
\li \; $\left\{ \bar U_{R_{e,k}}, \Di{\bar U_{R_{e,k}}} \right\}\leftarrow$ read-off from the ${\mathbb R}$ and $\Di{{\mathbb R}}$;
\li \; $\left\{ \bar U_{P_{k|k}}, \Di{\bar U_{P_{k|k}}} \right\}\leftarrow$ read-off from the ${\mathbb R}$ and $\Di{{\mathbb R}}$;
\li \; $\left\{ \bar K_{k}^u,  \Di{\bar K_{k}^u} \right\} \: \leftarrow$ read-off from the post-arrays ${\mathbb R}$ and $\Di{{\mathbb R}}$; \label{diff:f:pre:array2}
\li \; Find $e_{k}$ from line~\ref{lti:ud:f:ek} of Algorithm~1a and $\bar e_{k}=\bar U_{R_{e,k}}^{-1}e_{k}$; \label{eq:new:1}
\li \; Calculate $\Di{\bar e_{k}}  = -\bar U_{R_{e,k}}^{-1} \DDi{\bar U_{R_{e,k}}} \bar U_{R_{e,k}}^{-1} \bar e_{k}$
\zi \; $- \bar U_{R_{e,k}}^{-1}\left[\DDi{Z}\hat \alpha_{k|k-1}+Z\DDi{\hat \alpha_{k|k-1}}+\DDi{\beta}x_k\right]$; \label{diff:lti:f:ek}
\li \; Find {\it a posteriori} estimate $\hat \alpha_{k|k} = \hat \alpha_{k|k-1} + \bar K_{k}^u \bar e_k$;
\li \; Find $\Di{\hat \alpha_{k|k}}  = \Di{\hat \alpha_{k|k-1}}+\DDi{\bar K_{k}^u}\bar e_k + \bar K_{k}^u \DDi{\bar e_k}$. \label{diff:lti:f:X}
\end{codebox}

\begin{Prop} \label{proposition:3}
The filters' sensitivities of the UD-based LTI MIMO KF (Algorithm~1a) are calculated truly within Algorithm~1b.
\end{Prop}
\begin{proof}
Algorithm~1b computes simultaneously the filtering quantities and filter's sensitivities (with respect to the unknown system's parameter $\theta \in {\mathbb R}^p$). In other words, it incorporates the filtering equations, which have a convenient block structure suitable for the application of \cite[Lemma~1]{2013:Tsyganova:IEEE}. Therefore, the above pseudocode (\textsc{Diff. UD}) can be used in computations of the post-arrays and their derivatives. Thus, the last step left is to implement this pseudocode with the related pre-arrays and their derivatives, whose computation is straightforward since all such values are available in advance. As output parameters, this scheme returns the requested filtering quantities and their  sensitivities as well. Note that Lemma~1 and the corresponding pseudocode have been already proven and justified in the cited paper. All this verifies immediately the sensitivity calculation technique presented in lines~\ref{diff:p:pre:array1}--\ref{diff:p:pre:array2} and~\ref{diff:f:pre:array1}--\ref{diff:f:pre:array2} of Algorithm~1b.

Lastly, we need to prove the equations in lines~\ref{diff:lti:p:X}, \ref{diff:lti:f:ek} and~\ref{diff:lti:f:X} of this algorithm. With use of Jacobi's formula\footnote{$\Di{\left(A^{-1}\right)} = -A^{-1} \DDi{A} A^{-1}$ holds for any square nonsingular $A$.}, the direct differentiation of the equation in line~\ref{lti:ud:p:X} yields the state sensitivity formula
\begin{align*}
\Di{\hat \alpha_{k|k-1}} & = \Di{\left(\overline{T} \hat \alpha_{k-1|k-1}\right)} + \Di{\left(\overline{B}x_{k-1}\right)} + \Di{\left(SH^{-1}y_{k-1}\right)} \\
& = \DDi{\overline{T}}\hat \alpha_{k-1|k-1} + \overline{T}\DDi{\hat \alpha_{k-1|k-1}}+ \DDi{\overline{B}}x_{k-1}  \\
& + \DDi{S}H^{-1}y_{k-1} - SH^{-1}\DDi{H}H^{-1}y_{k-1},
\end{align*}
which is precisely that given in line~\ref{diff:lti:p:X} of Algorithm~1b. The above differentiation has taken into account that $\Di{y_k} = 0$ and $\Di{x_k} = 0$ because the observations $y_k$ and explanatory variables (or controllable inputs) $x_k$  do not depend on the parameters (i.e. their realizations are independent of variations in $\theta$). In the same way, the equations in lines~\ref{diff:lti:f:ek} and~\ref{diff:lti:f:X} of Algorithm~1b are justified by differentiating the formulas in lines~\ref{lti:ud:f:ek}, \ref{lti:ud:f:X} of the KF summarized in Algorithm~1a, i.e.
\begin{align*}
\Di{\bar e_k} & = \Di{\left(\bar U_{R_{e,k}}^{-1}e_k\right)} = -\bar U_{R_{e,k}}^{-1}\!\! \DDi{\bar U_{R_{e,k}}} \bar U_{R_{e,k}}^{-1} e_k + \bar U_{R_{e,k}}^{-1}\!\DDi{e_k}
\end{align*}
where $e_k = y_k-Z\hat \alpha_{k|k-1}-\beta x_k$ and, hence,
\begin{align*}
\Di{\bar e_k} & = -\bar U_{R_{e,k}}^{-1} \!\DDi{\bar U_{R_{e,k}}} \bar U_{R_{e,k}}^{-1} e_k \!+ \bar U_{R_{e,k}}^{-1}\!\Di{\!\left(y_k\!-\!Z\hat \alpha_{k|k-1}\!-\!\beta x_k\right)} \\
 & = -\bar U_{R_{e,k}}^{-1} \DDi{\bar U_{R_{e,k}}} \bar U_{R_{e,k}}^{-1} e_k \\
 & \phantom{=} + \bar U_{R_{e,k}}^{-1}\left[-\DDi{Z}\hat \alpha_{k|k-1}- Z\DDi{\hat \alpha_{k|k-1}}-\DDi{\beta} x_k\right].
\end{align*}
This proves line~\ref{diff:lti:f:ek} in our Algorithm~1b. Further, the equation in line~\ref{lti:ud:f:X} of Algorithm~1a is differentiated to the form
\begin{align*}
\Di{\hat \alpha_{k|k}}  & =  \Di{\hat \alpha_{k|k-1}}+\DDi{\bar K_{k}^u}\bar e_k + \bar K_{k}^u \DDi{\bar e_k},
\end{align*}
which validates line~\ref{diff:lti:f:X} in Algorithm~1b and completes our proof.
\end{proof}

Next, we derive a similar result for the UD-based PKF technique.
\begin{codebox}
\Procname{{\bf Algorithm~2b}. $\proc{Differentiated UD-based PKF}$}
\zi \textsc{Initialization:} ($k=0$, $i=\overline{1,p}$)
\li \; Repeat initialization of Algorithm~2a in lines~1--4;
\li \; Compute $\Di{\overline{T}}$, $\Di{\overline{B}}$, $\Di{\overline{Q}}$, $\Di{\bar U_{\Pi_0}}$, $\Di{D_{\Pi_0}}$, $\Di{\bar \alpha_0}$;
\li \; Set $\Di{\bar U_{P_{0|0}}}\!\! = \Di{\bar U_{\Pi_0}}$, $\Di{D_{P_{0|0}}}\!\! = \Di{D_{\Pi_0}}$, $\Di{\hat \alpha_{0|0}} = \Di{\bar \alpha_0}$;
\zi \textsc{Time Update}: ($k=\overline{0,N-1}$, $i=\overline{1,p}$) \Comment{\small\textsc{Priori estimation}}
\li \; Find {\it a priori} estimate $\hat \alpha_{k+1|k}$ from line~\ref{pkf:ud:p:X} of Algorithm~2a;
\li \; Compute $\Di{\hat \alpha_{k+1|k}} = \DDi{\overline{T}}\hat \alpha_{k|k} + \overline{T}\DDi{\hat \alpha_{k|k}}$
\zi \; $+\DDi{\overline{B}} y_{k-1} + \DDi{S }H^{-1}y_{k} - S H^{-1}\DDi{H}H^{-1}y_{k}$;
\li \; Build the pre-arrays ${\mathbb A}$, ${\mathbb D}_A$ in line~\ref{pkf:ud:p:P} of Algorithm~2a;
\li \; Find the pre-arrays' derivatives $\Di{\mathbb A}$ and $\Di{{\mathbb D}_A}$;
\li \; Get $\left[{\mathbb R}, {\mathbb D}_R, \Di{{\mathbb R}}, \Di{{\mathbb D}_R} \right]\leftarrow$\verb"Diff. UD"(${\mathbb A}$, ${\mathbb D}_A$, $\Di{\mathbb A}$, $\Di{{\mathbb D}_A}$);
\li \; $\left\{ D_{P_{k+1|k}}, \: \Di{D_{P_{k+1|k}}} \right\} \leftarrow$ read-off from the ${\mathbb D}_R$ and $\Di{{\mathbb D}_R}$;
\li \; $\left\{\bar U_{P_{k+1|k}}, \: \Di{\bar U_{P_{k+1|k}}} \right\}\leftarrow$ read-off from the ${\mathbb R}$ and $\Di{{\mathbb R}}$;
\zi \textsc{Measurement Update}: \Comment{\small\textsc{Posteriori estimation}}
\li \; Build the pre-arrays ${\mathbb A}$, ${\mathbb D}_A$  in line~\ref{pkf:ud:f:P} of Algorithm~2a;
\li \; Find the pre-arrays' derivatives $\Di{\mathbb A}$ and $\Di{{\mathbb D}_A}$; \label{start:eq:new}
\li \; Get $\left[{\mathbb R}, {\mathbb D}_R, \Di{{\mathbb R}}, \Di{{\mathbb D}_R} \right]\leftarrow$\verb"Diff. UD"(${\mathbb A}$, ${\mathbb D}_A$, $\Di{\mathbb A}$, $\Di{{\mathbb D}_A}$);
\li \; $\left\{ D_{R_{e,k+1}}, \Di{D_{R_{e,k+1}}} \right\}\leftarrow$ read-off from the ${\mathbb D}_R$ and $\Di{{\mathbb D}_R}$;
\li \; $\left\{ D_{P_{k+1|k+1}}, \Di{D_{P_{k+1|k+1}}} \right\}\leftarrow$ read-off from the ${\mathbb D}_R$, $\Di{{\mathbb D}_R}$;
\li \; $\left\{ \bar U_{R_{e,k+1}}, \Di{\bar U_{R_{e,k+1}}} \right\}\leftarrow$ read-off from the ${\mathbb R}$, $\Di{{\mathbb R}}$; \label{diff:pkf:start}
\end{codebox}
\begin{codebox}
\setlinenumberplus{diff:pkf:start}{1}
\li \; $\left\{ \bar U_{P_{k+1|k+1}}, \Di{\bar U_{P_{k+1|k+1}}} \right\}\leftarrow$ read-off from the ${\mathbb R}$, $\Di{{\mathbb R}}$;
\li \; $\left\{ \bar K_{k+1}^u,  \Di{\bar K_{k+1}^u} \right\} \: \leftarrow$ read-off from the ${\mathbb R}$ and $\Di{{\mathbb R}}$;
\li \; Find $e_{k+1}$ from line~\ref{pkf:ud:f:ek} of Algorithm~2a, $\bar e_{k+1}=\bar U_{R_{e,k+1}}^{-1}e_{k+1}$;
\li \; Calculate $\Di{\bar e_{k+1}}  = -\bar U_{R_{e,k+1}}^{-1} \DDi{\bar U_{R_{e,k+1}}} \bar U_{R_{e,k+1}}^{-1} \bar e_{k+1}$
\zi \; $- \bar U_{R_{e,k+1}}^{-1}\left[\DDi{Z}\hat \alpha_{k+1|k}+Z\DDi{\hat \alpha_{k+1|k}}+\DDi{\beta}y_{k}\right]$;
\li \; Find {\it a posteriori} estimate $\hat \alpha_{k+1|k+1} = \alpha_{k+1|k}+\bar K_{k+1}^u \bar e_{k+1}$;
\li \; $\Di{\hat \alpha_{k+1|k+1}}  = \Di{\hat \alpha_{k+1|k}}+\DDi{\bar K_{k+1}^u}\bar e_{k+1} + \bar K_{k+1}^u \DDi{\bar e_{k+1}}$.
\end{codebox}

\begin{Prop} \label{proposition:4}
The filters' sensitivities of the UD-based PKF equations (Algorithm~2a) are calculated truly within Algorithm~2b.
\end{Prop}
\begin{proof} Having extended the methodology utilized in the proof of Proposition~\ref{proposition:3} to the UD-based PKF equations (Algorithm~2a), one derives easily the formulas summarized in Algorithm~2b.
\end{proof}

\section{Numerical experiments} \label{numerical:experiments}

For a practical examination of the new UD-based adaptive schemes, we consider several state estimation scenarios from a rather well-conditioned stochastic system to its moderately and strongly ill-conditioned counterparts. Here, we focus on testing the LTI MIMO gradient-based estimator in Algorithm~1b, but the PKF estimator in Algorithm~2b allows similar results and conclusions to be observed.

\begin{exmp}\label{ex:2}
Consider the LTI MIMO specification in~\eqref{eq:lti:model} with
\begin{align*}
T & =
\begin{bmatrix}
1 & 1 & 0.5 &  0.5 \\
0 & 1 & 1 & 1 \\
0 & 0 & 1 & 0 \\
0 & 0 & 0 & 0.606
\end{bmatrix},
 & Z & =
\begin{bmatrix}
1 & 1 & 1 & 1\\
1 & 1 & 1 & 1+\delta
\end{bmatrix}
\end{align*}
and $Q = {\rm diag}\{0,0,0, 0.63 \cdot 10^{-2}\}$, $H  = \theta^2\delta^2 I_2$, $B=0$, $\beta=0$, $S=0$. The initial values are: $\bar \alpha_0 =0$ and $\Pi_0 =  \theta^2 I_4$. Here, $I_2$ and $I_4$ stand for the identity matrices of sizes~$2$ and~$4$, respectively. The ill-conditioning parameter $\delta$ is used for provoking numerical instabilities and boosting the roundoff error. Then, the system's parameter $\theta$ is to be estimated together with the dynamic state $\alpha_k$ of this model.
\end{exmp}

\begin{table*}
\renewcommand{\arraystretch}{1.3}
{\scriptsize
\caption{The performance results of the adaptive schemes for the predefined parameter value $\theta^* =3$ in Example~1} \label{MC-estimators}
\centering
\begin{tabular}{c||c|c|c||c|c|c||c|c|c||c|c|c}
\hline
&  \multicolumn{6}{c||}{\bf Conventional parameter estimator implementations} &  \multicolumn{6}{c}{\bf $UDU^\top$-factorization-based parameter estimator implementations} \\
\cline{2-13}
&  \multicolumn{3}{c||}{\bf Numerically approximated score} &
\multicolumn{3}{c||}{\bf Analytically computed score} & \multicolumn{3}{c||}{\bf Numerically approximated score} & \multicolumn{3}{c}{\bf Analytically computed score}\\
\cline{2-13}
$\delta$ &  Mean &  RMSE, MAPE & Time &  Mean & RMSE, MAPE & Time\,(Benefit) &  Mean & RMSE, MAPE & Time &  Mean & RMSE, MAPE & Time\,(Benefit) \\
\hline
\hline
$ 10^{\pm 0\phantom{0}}$ & 2.96 &  0.13, 4.18\% & 1.19 & 2.96 & 0.13, 4.18\% & 1.16 (2.5\%) & 2.96 & 0.13, 4.18\% & 2.66 & 2.96 & 0.13, 4.18\% & 2.54 (4.7\%) \\
$ 10^{-1\phantom{0}}$ & 2.94 &  0.14, 4.03\% & 1.23 & 2.94 & 0.14, 4.03\% & 1.20 (2.5\%) & 2.94 & 0.14, 4.03\% & 2.49 & 2.94 & 0.14, 4.03\% & 2.41 (3.3\%) \\
\hline
$ 10^{-2\phantom{0}}$ & 2.97 & 0.16, 4.53\% & 2.83 & 2.97 & 0.16, 4.53\%  & 1.15 (146\%) & 2.97 & 0.16, 4.53\% & 2.76 & 2.97 & 0.16, 4.53\% & 2.58 (6.9\%) \\
$ 10^{-3\phantom{0}}$ & 2.98 & 0.11, 3.23\%  & 2.94 & 3.02 &  0.14, 4.34\% & 1.46 (100\%) & 3.02 & 0.14, 4.30\% & 2.67 & 3.02 & 0.14, 4.30\% & 2.51 (6.4\%) \\
$ 10^{-4\phantom{0}}$ & 2.50 & 0.71, 21.0\%  & 3.05 & 4.31 &  2.08, 43.8\% & 8.80 (-65\%) & 2.94 & 0.20, 6.06\% & 2.84 & 2.94 & 0.20, 6.06\% & 2.59 (9.4\%) \\
$ 10^{-5\phantom{0}}$ & 2.33 & 1.06, 27.1\%  & 3.41 & 7.77 &  6.01, 185\% & 0.77 (342\%) & 2.96 & 0.23, 7.19\% & 3.78 & 2.96 & 0.23, 7.19\% & 2.44 (54\%) \\
\hline
$ 10^{-6\phantom{0}}$ & 1.00 & 2.00, 66.6\%  & 2.57 & 5.52 & 6.33, 187\% & 1.06 (141\%) & 2.75 & 0.30, 8.56\% & 6.64 & 2.75 & 0.30, 8.56\% & 2.65 (149\%) \\
$ 10^{-7\phantom{0}}$ & 1.00 & 2.00, 66.6\%  & 2.22 & 1.00 & 2.00, 66.6\% & 0.67 (227\%) & 2.91 & 0.25, 7.00\% & 7.08 & 2.90 & 0.25, 7.06\% & 2.44 (189\%) \\
$ 10^{-8\phantom{0}}$ & - & - & - & - & - & - & 2.99 & 0.20, 5.82\% & 6.31 & 3.00 & 0.19, 5.48\% & 2.52 (150\%) \\
$ 10^{-9\phantom{0}}$ & - & - & - & - & - & - & 2.78 & 0.51, 13.56\% & 6.78 & 2.85 & 0.24, 7.04\% & 2.85 (137\%) \\
$ 10^{-10\phantom{0}}$ & - & - & - & - & - & - & 2.06 & 1.03, 31.29\% & 6.20 & 2.94 & 0.30, 9.11\% & 4.78 (29\%) \\
$ 10^{-11\phantom{0}}$ & - & - & - & - & - & - & 1.58 & 1.48, 47.19\% & 8.28 & 2.93 & 0.27, 7.35\% & 6.90 (20\%) \\
$ 10^{-12\phantom{0}}$ & - & - & - & - & - & - & 1.01 & 1.98, 66.02\% & 6.05 & 3.34 & 0.45, 11.54\% & 28.58 (-78\%) \\
\hline
\end{tabular}
}
\end{table*}

The adaptive filtering schemes listed below are applied for the estimation of the predefined (i.e. ``true'') parameter value $\theta^* = 3$ from $N=100$ simulated measurements. For a fair comparison, each adaptive strategy utilizes precisely the same sampled data $\{ y_k\}$, the same initial value $\hat \theta^{(0)} = 1$ set for the optimization method implemented and the same optimization code, which is taken to be the MATLAB built-in function \verb"fmincon".
In Example~\ref{ex:2}, we examine the adaptive filtering strategies with the numerically approximated gradient (i.e. with the optimizer's option 'GradObj' set to 'off') and with the analytically computed score (where the optimizer's option 'GradObj' is 'on'). The estimators in use are the following ones:
\begin{itemize}
\item the conventional LTI MIMO computing the log LF by Algorithm~1 and \verb"fmincon" with the gradient found numerically;
\item the conventional LTI MIMO computing the log LF by Algorithm~1 and \verb"fmincon" with the gradient found analytically by differentiation of the algorithm's equations (see details in~\cite{2008:Wills});
\item the novel UD-based LTI MIMO computing the log LF by Algorithm~1a and \verb"fmincon" with the gradient found numerically;
\item the novel UD-based LTI MIMO computing the log LF and the analytical gradient utilized in \verb"fmincon" by Algorithm~1b.
\end{itemize}

Our numerical experiments are fulfilled for estimating the parameter $\hat \theta^*$ within $M=100$ Monte Carlo runs, which allow  {\it a posterior} mean of derived parameter estimates, the root mean squared error (RMSE) and the mean absolute percentage error (MAPE) to be evaluated. For each adaptive scheme, we also assess the average CPU time (in sec.) and calculate the computational benefit (in $\%$) of using the analytically computed log LF gradient instead of its numerical approximation in the optimization method. For detecting the most reliable technique among the listed ones, we repeat our numerical experiments for various $\delta$'s such that this ill-conditioning parameter tends to the machine precision limit of MATLAB, $\delta \to \epsilon_{roundoff}$. We intend for assessing degradations of the methods due to roundoff, which stems from the inversion of increasingly ill-conditioned matrix
$R_{e,k}$; see the third reason of the KF ill-conditioning in~\cite[p.~288]{2015:Grewal:book}.

{\it The well-conditioned scenario}. Here, we assess the performance of the listed parameter estimators in the situation when the value of $\delta$ is reasonably large. Having analyzed the results collected in the first and second rows of Table~1, we see that all the adaptive schemes under examination solve equally well Example~\ref{ex:2}, i.e. with the same accuracies. Besides, the computed RMSE and MAPE are small, i.e. these adaptive techniques maintain high estimation quality. In other words, accuracies achieved by the examined strategies are high and insensitive to the covariance matrix factorization and the way of the score calculation. In addition, the numerical results presented substantiate our theoretical derivations in Sec.~\ref{main:result} and justify the mathematical equivalence of the conventional approach to the filter derivative computation and our novel UD-based ``differentiated'' algorithms. Also, the strategy with the analytically computed gradient reduces the execution time of the entire adaptive filtering schemes. This conclusion holds for the conventional parameter estimators and their UD-based versions as well. The CPU benefit of using the analytically computed score instead of its numerical approximation is about $3-4\%$ on average, but this run time reduction will change dramatically when the problem becomes increasingly ill-conditioned.

{\it The moderately ill-conditioned scenario}. Increasing the model's ill-conditioning up to $\delta = 10^{-5}$ facilitates our assessment of all the parameter estimators under the mildly ill-conditioning case. Within this scenario, we observe a more significant CPU time profit for the methods with the analytically computed score. At $\delta = 10^{-2}$ and $\delta = 10^{-3}$, all the adaptive schemes estimate similarly the unknown parameter $\theta$. Indeed, their estimation errors are small and roughly the same regardless of using or not the $UDU^\top$ covariance factorization. However, we see that, for providing an accurate parameter estimation, the optimizer grounded in the conventional implementation and with the numerically approximated score (these data are presented in the first panel of Table~1) requires about twice of the CPU time utilized by the conventional implementation with the analytically computed score (these data are exposed in the second panel of Table~1). Meanwhile, the UD-based strategies are equally accurate and have almost the same computational cost, as shown in our third and fourth panels.

Next, for $\delta \le 10^{-4}$, the roundoff errors can affect the estimators' performance, considerably. It is clearly seen that the proposed UD-based adaptive schemes are robust with respect to roundoff errors regardless of the score calculation way utilized. Having analyzed the outcomes presented for $\delta = 10^{-4}$ and $\delta = 10^{-5}$ in the last two panels of Table~1, we observe the same values of the system's parameter $\hat \theta$ estimated by our UD-based methods with the equally small RMSE and MAPE committed. However, at $\delta = 10^{-5}$, the UD-based optimization scheme with the analytically computed log LF gradient is about $1.5$ times faster than that utilized the gradient approximated numerically. Although the execution time has been increased, our novel UD-based adaptive strategies maintain the high estimation quality and supply users with the reliable parameter identification means suitable for practical use. Meanwhile, the results derived by the conventional implementations are not so optimistic; see the first and second panels of Table~1. In fact, we observe even the divergence of non-factorized adaptive schemes regardless of the way of the score computation implemented. Note that their estimation errors become large with the outcome MAPEs being $\approx 20\%$ and $\approx 40\%$  for the numerically and analytically calculated scores at $\delta = 10^{-4}$, respectively. These are clear signs of the poor parameter estimation fulfilled. At $\delta = 10^{-5}$, the situation is further worsen. The conventional approach with the analytically computed gradient is not capable of solving the stated parameter estimation problem because its committed MAPE exceeds $100\%$, that illustrates the failure of this adaptive scheme. Due to the roundoff errors, the filter itself and its ``differentiated'' part provide highly inaccurate values of the log LF and its gradient. Then, these poor values are utilized by the optimizer and, eventually, contribute to the poor parameter estimates calculated. The conventional approach with the numerically approximated score works better in this situation because the log LF is the unique source of the incorrect information, here. However, its MAPE is about $27\%$, that also implies the poor performance of this adaptive technique.

{\it The strongly ill-conditioned scenario}.  When the ill-conditioning parameter $\delta$ becomes extremely small, we will observe the ultimate failure of the classical adaptive filtering schemes. Indeed, for $\delta = 10^{-6}$ and $\delta = 10^{-7}$, the data of the first two panels in Table~1 expose the extremely poor performance. Sometimes, these even return the initial value $\theta^{(0)}=1$ as their output. Furthermore, at $\delta\ge 10^{-8}$, the roundoff errors destroy fully the standard parameter estimators regardless of the way for calculating the score. In Table~1, we use the dash ``-'' for indicating the situation when the estimator either cannot be run or these are suddenly interrupted and terminated. Meanwhile, our novel UD-based schemes work well even under this extreme ill-conditioning. In fact, the new methods produce the accurate estimates with their MAPEs about $6-7\%$ for any $\delta\le 10^{-9}$. Nevertheless, the version with the analytically computed score demands less CPU time in comparison to that computing this score, numerically. In addition, our further decrease of the parameter $\delta$ reduces essentially the performance of the UD-based estimator with the numerically approximated score; see the third panel in Table~1. Indeed, at $\delta = 10^{-10}$, its MAPE becomes about $30\%$ compared to $9\%$ exhibited by the UD-based technique using the newly-proposed analytic way for computing the score. Having completed our numerical experiments at all $\delta\le 10^{-12}$, we see that the UD-based method with the numerically approximated score fails and produces the unacceptable results in the last scenario. The third panel in Table~1 shows that this version is fast, but extremely poor. At the same  $\delta = 10^{-12}$, the UD-based estimator with the analytically computed score produces the rather accurate estimate with its MAPE about $11\%$, but it requires a significant computational effort for reaching this result because the log LF and its gradient are both affected severely by the roundoff errors involved.

In conclusion, we see that our novel UD-based adaptive filtering schemes are mathematically equivalent to their classical counterparts, but these are more robust to roundoff errors and provide high quality parameter estimations even in strongly ill-conditioned scenarios. All this makes them the methods of choice in parameter estimation tasks.

\subsection{Real data application} \label{empirical:results}

The proposed efficient and reliable gradient-based UD-factorized adaptive filtering strategy may be appealing to practitioners working in econometrics. In particular, the extra assumption of time-varying coefficients underlying many classical econometrics models yields often state-space model structures of the form~\eqref{eq:lti:model} or~\eqref{eq:pkf:model}. In Sec.~\ref{empirical:results}, we explore the so-called {\em test for evolving efficiency} (TEE) that is, in fact, the classical GARCH-in-Mean(1,1) model with time-varying regression coefficients. Such models are used for recovering a weak-form market efficiency process, as suggested in~\cite{1999:Zalewska}.

The notion of market efficiency was introduced in the late 1960s with the efficient financial market defined as follows~\cite{1970:Fama}: ``as one in which security prices at any time always fully reflect available information'', i.e. there is no predictable profit opportunity based on the past movement of asset prices.
The foundational assumption of this market efficiency justifies the use of Markov processes in financial industry. It also underlies the most fundamental models in finance and econometrics~\cite{1972:Black}. A number of statistical tests have been derived for a weak-form market efficiency estimation. It is often examined in the absolute sense where the market efficiency is assumed to be unchanged throughout the examination period. Since the end of 1990s, the so-called {\it adaptive} market hypothesis gains popularity claiming that this property is not stable over time~\cite{2004:Lo}. It is particularly true for young (``infant'') markets, but it holds for emerging markets as well where some anomalies in the efficiency process might happen.
Below, we consider the TEE developed in~\cite{1999:Zalewska}, which is widely used nowadays.
The cited model is given as follows:
\begin{align}
y_k           & =  \mu_k + \varepsilon_k,   \quad  \varepsilon_k \sim {\cal N}(0, h_k), \label{tee:yk} \\
\mu_k & = \beta_{0,k} + \beta_{1,k} y_{k-1} + \delta h_k, \label{tee:mean} \\
h_k  & =  a_0 + a_1 (y_{k-1}-\mu_{k-1})^2 + b_1 h_{k-1}, \label{garchm:hk} \\
\beta_{i,k}   & = \beta_{i,k-1} + \eta_{i,k-1}, \: (i=0, 1) \quad  \eta_{i,k} \sim {\cal N}(0, \sigma_{i}^2) \label{tee:beta}
\end{align}
where $a_0 >0$, $a_1 \ge 0$, $b_1 \ge 0$ and $a_1 + b_1 <1$ ensures that the GARCH(1,1) process in equation~\eqref{garchm:hk} is stationary. The coefficient $\delta$ in the mean equation~\eqref{tee:mean} measures the volatility-in-mean effect~\cite{2009:Koopman}. The sequence $\{y_k\}_{k=0}^{N}$ contains the log-returns of asset prices and the process $\{h_k\}_{k=0}^{N}$ is the unknown variance process, which needs to be estimated. Following~\cite{1999:Zalewska}, the {\it evolution} of the slope coefficient $\beta_{1,k}$ in equation~\eqref{tee:mean} reflects changes in the weak-form market efficiency. It is also applied for detecting time periods where the market is weak-form efficient, which are identified by the condition $\beta_{1,k} = 0$.

The TEE methodology examines the EURO STOXX 50 index, which covers the largest and most liquid $50$ stocks in $18$ European countries. The time period to be analyzed is from 1 February, 1998 till 30 December, 2017\footnote{The data set can be freely downloaded at https://finance.yahoo.com.}.
For estimating the dynamic state $\alpha_k = [h_{k},\beta_{0,k},\beta_{1,k}]^\top$ together with the unknown system's parameters $\theta = [a_0,a_1,b_1,\delta,\sigma_0,\sigma_1]$ from the observed return series $y_k = 100 (\ln S_k - \ln S_{k-1})$ where $\{ S_k \}$ are the closing prices, the adaptive filtering technique is to be applied. Again, we use the MATLAB built-in function \verb"fmincon" as our optimization routine with Algorithm~1b  implemented for computing the log LF and its gradient as well, which are then passed to the optimizer. Algorithm~1b will be trivially accommodated to model (\ref{tee:yk})--(\ref{tee:beta}) if it is casted into the LTI MIMO state-space form~\eqref{eq:lti:model} with $T := {\rm diag}\{b_1,1,1\}$, $Z := [\delta, 1, y_{k-1}]$, $B:=[a_0,a_1;0,0;0,0]$, $\beta:=0$ and with the noise covariances $Q := {\rm diag} \{0,\sigma_{0}^2, \sigma_{1}^2 \} $ and $H := \hat h_{k|k-1}$, respectively. Following~\cite{1991:Hall:note,1991:Hall}, the  explanatory variables are: $x_k := [1, \hat \varepsilon_{k}^2]^\top$ where the residual is defined by $\hat \varepsilon_{k} := y_k- \hat \beta_{0,k|k-1} - \hat \beta_{1,k|k-1} y_{k-1} - \delta \hat h_{k|k-1}$ at time $t_k$ and $\alpha_{k|k-1} := [h_{k|k-1},\beta_{0,k|k-1},\beta_{1,k|k-1}]^\top$, which is the {\it a priori} state estimate obtained by the underlying filtering method.

\begin{figure}
\begin{center}
\includegraphics[width=0.5\textwidth]{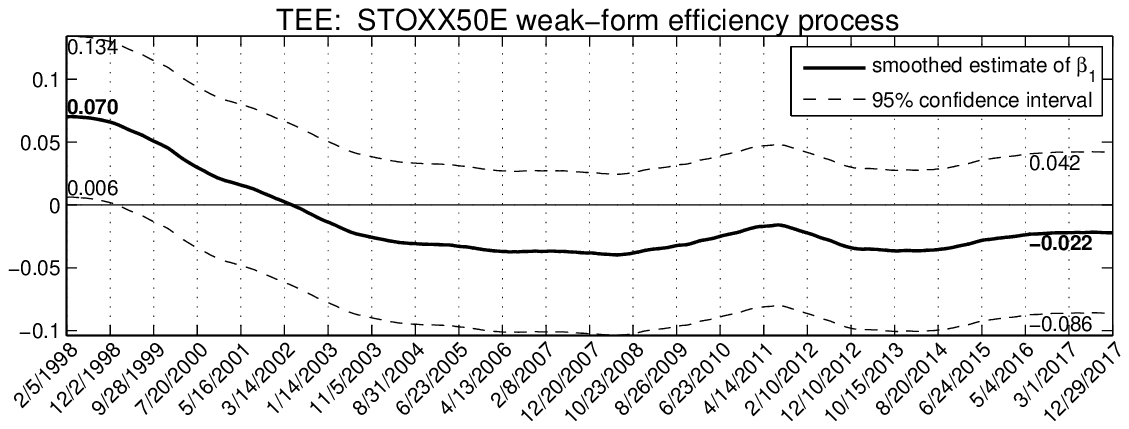}  \\  
\caption{The test for evolving efficiency of the EURO STOXX 50 index.}
\label{fig:1}
\end{center}
\end{figure}

Fig.~\ref{fig:1} illustrates the TEE outcome in the EURO-STOXX-50-index-based scenario. We observe the weak-form efficiency (with $95$\% confidence level) for the entire period of the study, except the short period in 1998 where the lower bound of the confidence interval just touches the critical zero value. Indeed, the parameter $\beta_{1,k}$ estimated is $0.07$ with the lower bound of the constructed $95$\% confidence interval being $0.006$ in the beginning of 1998. However, this market shows the quick tendency to the weak-form efficiency, which is achieved with the $95$\% confidence level since December of 1998. Then, it remains efficient until the end of the examination time in December of 2017. In addition, it is important to remark that the efficiency is not a constant value in the EURO-STOXX-50-index-based scenario under consideration. Eventually, we observe its significant changes, which also support and validate the above adaptive market hypothesis.

\section{Concluding remarks} \label{conclusion}

In this paper, the gradient-based adaptive strategy for simultaneous hidden dynamic state and system's parameter estimations is derived within general state-space models structures. These methods are developed in terms of the $UDU^\top$ matrix decomposition. Out novel UD-based technique with the analytic filter's sensitivity evaluation is computationally efficient and reliable for treating system identification tasks, including ill-conditioned scenarios.
The proposed algorithms require the calculation of the filter derivatives, only. Alternatively, one has to utilize the forward (filtering) and backward (smoothing) pass~\cite{1992:Koopman}. However, it is presently unknown how to implement the $UDU^\top$-factorization-based strategy for calculating derivatives of the filter equations and of the smoother ones as well. This seems to be an interesting topic of a future research.

\bibliographystyle{IEEEtran}

\end{document}